\documentclass[12pt,twoside,a4paper]{article}
\usepackage{times}
\usepackage{bm}
\usepackage{braket}
\usepackage{graphicx}
\usepackage{theorem}
\usepackage{amsmath}
\usepackage{amssymb,mathrsfs}
\usepackage{latexsym}
\usepackage[numbers]{natbib}
\usepackage{algorithm}
\usepackage{algpseudocode}
\usepackage{caption}
\newfloat{procedure}{H}{loa}
\floatname{procedure}{Procedure} 
\usepackage[flushmargin]{footmisc}

\theorembodyfont{\itshape}
\newtheorem{thm}{Theorem}[section]
\newtheorem{lem}{Lemma}[section]
\newtheorem{prop}{Proposition}[section]
\newtheorem{coro}{Corollary}[section]

\theorembodyfont{\rmfamily}
{\bf}{\rm}
{\bf}{\rm}

\newtheorem{hm-cond}{Condition}
\newtheorem{rem}{Remark}[section]{\itshape}{\rmfamily}
{\itshape}{\rmfamily}
\newenvironment{proof}{\noindent{\it Proof.~~}}{\medskip}

\makeatletter
\def\eqnarray{\stepcounter{equation}\let\@currentlabel=\theequation
\global\@eqnswtrue
\global\@eqcnt\z@\tabskip\@centering\let\\=\@eqncr
$$\halign to \displaywidth\bgroup\@eqnsel\hskip\@centering
  $\displaystyle\tabskip\z@{##}$&\global\@eqcnt\@ne 
  \hfil$\;{##}\;$\hfil
  &\global\@eqcnt\tw@ $\displaystyle\tabskip\z@{##}$\hfil 
   \tabskip\@centering&\llap{##}\tabskip\z@\cr}
\makeatother
\makeatletter
    \renewcommand{\theequation}{%
    \thesection.\arabic{equation}}
    \@addtoreset{equation}{section}
  \makeatother
\makeatletter
\def\Left#1#2\Right{\begingroup%
   \def\ts@r{\nulldelimiterspace=0pt \mathsurround=0pt}%
   \let\@hat=#1%
   \def\sht@im{#2}%
   \def\@t{{\mathchoice{\def\@fen{\displaystyle}\k@fel}%
          {\def\@fen{\textstyle}\k@fel}%
          {\def\@fen{\scriptstyle}\k@fel}%
          {\def\@fen{\scriptscriptstyle}\k@fel}}}%
   \def\g@rin{\ts@r\left\@hat\vphantom{\sht@im}\right.}%
   \def\k@fel{\setbox0=\hbox{$\@fen\g@rin$}\hbox{%
      $\@fen \kern.3875\wd0 \copy0 \kern-.3875\wd0%
      \llap{\copy0}\kern.3875\wd0$}}%
      \def\pt@h{\mathopen\@t}\pt@h\sht@im%
      \Right}%
\def\Right#1{\let\@hat=#1%
   \def\st@m{\mathclose\@t}%
   \st@m\endgroup}
\makeatother

\newcommand{\vc}{\bm}
\makeatletter
\DeclareRobustCommand\widecheck[1]{{\mathpalette\@widecheck{#1}}}
\def\@widecheck#1#2{%
    \setbox\z@\hbox{\m@th$#1#2$}%
    \setbox\tw@\hbox{\m@th$#1%
       \widehat{%
          \vrule\@width\z@\@height\ht\z@
          \vrule\@height\z@\@width\wd\z@}$}%
    \dp\tw@-\ht\z@
    \@tempdima\ht\z@ \advance\@tempdima2\ht\tw@ \divide\@tempdima\thr@@
    \setbox\tw@\hbox{%
       \raise\@tempdima\hbox{\scalebox{1}[-1]{\lower\@tempdima\box
\tw@}}}%
    {\ooalign{\box\tw@ \cr \box\z@}}}
\makeatother

\newcommand{\ol}{\overline}

\newcommand{\wt}{\widetilde}
\newcommand{\wh}{\widehat}
\newcommand{\bv}{\breve}

\newcommand{\vertiii}[1]%
{{\left\vert\kern-0.25ex\left\vert\kern-0.25ex\left\vert #1 
 \right\vert\kern-0.25ex\right\vert\kern-0.25ex\right\vert}}
\newcommand{\down}[2]{\smash{\lower#1\hbox{#2}}}
\newcommand{\up}[2]{\smash{\lower-#1\hbox{#2}}}

\newcommand{\dm}{\displaystyle}
\newcommand{\qed}{\hspace*{\fill}$\Box$}

\newcommand{\vmin}{\wedge}

\newcommand{\EE}{\mathsf{E}}
\newcommand{\PP}{\mathsf{P}}

\newcommand{\one}{\mbox{$1$}\hspace{-0.25em}{\rm l}}

\newcommand{\calB}{\mathcal{B}}

\newcommand{\calD}{\mathcal{D}}

\newcommand{\bbA}{\mathbb{A}}

\newcommand{\bbC}{\mathbb{C}}

\newcommand{\bbM}{\mathbb{M}}

\newcommand{\bbR}{\mathbb{R}}

\newcommand{\bbX}{\mathbb{X}}

\newcommand{\bbZ}{\mathbb{Z}}
\newcommand{\scA}{\mathscr{A}}
\newcommand{\scB}{\mathscr{B}}

\DeclareMathOperator*{\argmax}{arg\,max}

\newcommand{\re}{{\rm e}}

\newcommand{\varep}{\varepsilon}

\newcommand{\dd}[1]{\if#11 1\!\!1 
\else {\if#1C I\!\!\!C
\else {\if#1G I\!\!\!G 
\else {\if#1J J\!\!\!J 
\else {\if#1S S\!\!\!S
\else {\if#1Z Z\!\!\!Z
\else {\if#1Q O\!\!\!\!Q
\else I\!\!#1
\fi}
\fi}
\fi}
\fi} 
\fi} 
\fi} 
\fi} 

\makeatother

\begin{document}\thispagestyle{empty} 

\phantom{}\vspace{-10mm}

\hfill

{\Large{\bf
\begin{center}
Specific bounds for a probabilistically interpretable solution of the Poisson equation for general state-space Markov chains with queueing applications%
\footnote[1]{%
Submitted to Applied Probability Trust
}
%
%
%
%
\end{center}
}
}

\begin{center}
{
Hiroyuki Masuyama%
\footnote[2]{E-mail: masuyama@sys.i.kyoto-u.ac.jp}
}

\medskip

{\small
Department of Systems
Science, Graduate School of Informatics, Kyoto University\\
Kyoto 606-8501, Japan
}

\bigskip
\medskip

{\small
\textbf{Abstract}

\medskip

\begin{tabular}{p{0.85\textwidth}}
This paper considers the Poisson equation for general state-space Markov chains in continuous time. The main purpose of this paper is to present  specific bounds for the solutions of the Poisson equation for general state-space Markov chains. The solutions of the Poisson equation are unique in the sense that they are expressed in terms of a certain probabilistically interpretable solution (called the {\it standard solution}). Thus, we establish some specific bounds for the standard solution under the $f$-modulated drift condition (which is a kind of Foster-Lyapunov-type condition) and some moderate conditions. To demonstrate the applicability of our results, we consider the workload processes in two queues: MAP/GI/1 queue, and M/GI/1 queue with workload capacity limit.
\end{tabular}
}
\end{center}

\begin{center}
\begin{tabular}{p{0.90\textwidth}}
{\small

\medskip

}
\end{tabular}

\end{center}


\section{Introduction}\label{sec-intro}

In this paper, we consider an ergodic continuous-time Markov chain $\{X(t);t \in \bbR_+:=[0,\infty)\}$ with a topological state space $\bbX$ and extended generator $\scA$ (which is formally defined in the next section). Let $\{P^t;t\in\bbR_+\}$ denote the transition semigroup of the Markov chain $\{X(t)\}$, i.e.,
\begin{equation*}
P^t(x,\bbA) = \PP_x(X(t) \in \bbA),
\quad
t\in\bbR_+,\ x \in\bbX,\ \bbA \in \calB(\bbX),
\end{equation*}
where $\PP_x(\,\cdot\,)=\PP(\,\cdot\, |\, X(0)=x)$ and $\calB(\bbX)$ denotes the Borel $\sigma$-field on $\bbX$. For later use, we introduce some conventions. Let $\bbR$ denote the set of all real numbers. For any function $f:\bbX\to\bbR$, let $|f|$ denote a function $\bbX\to\bbR_+$ such that $|f|(x)=|f(x)|$ for all $x \in \bbX$. Let $\braket{\nu,f} = \int_{x\in\bbX}\nu(dx) f(x)$ for any measure $\nu$ on $\calB(\bbX)$ and any real Borel (measurable) function $f$ on $\bbX$, . 

In this paper, we consider the Poisson equation for the Markov chain $\{X(t)\}$:
\begin{equation}
-\scA h = g - \braket{\pi,  g},
\label{Poisson-EQ}
\end{equation}
where $g:\bbX \to \bbR$ is a given Borel function, and where $\pi$ is the invariant probability measure of $\{X(t)\}$. 
Poisson equation (\ref{Poisson-EQ}) and its variants appear in various studies on Markov chains \cite{Mako02}, such as the functional central limit theorem~(\cite{Glyn96}, \cite[Section 17.4]{Meyn09}), stochastic approximation algorithms~\cite{Mako92,Meti84}, perturbation analysis \cite{Cao97,Cao98,LiuYuan15}, and augmented truncation approximation~\cite{LiuYuan18-AAP,LLM2018,Masu17-JORSJ}. 

We now suppose that $|\braket{\pi,  g}| < \infty$. We then define $h^{(g)}$ as a function $\bbX\to\bbR$ such that
\begin{equation}
h^{(g)}(x) 
= \EE_x \! \left[ \int_0^{\tau_{\alpha}} g(X(t)) dt\right]
- \braket{\pi,  g} \EE_x[\tau_{\alpha}],\qquad x \in \bbX,
\label{defn-h^{(g)}}
\end{equation}
where $\tau_{\alpha}:=\inf\{t > 0: X(t) \in \alpha, X(t-) \not\in \alpha\}$ is the first return time to an {\it atom} $\alpha \in \calB(\bbX)$ (see Condition~\ref{cond-atom} below). The function $h^{(g)}$ is a solution of Poisson equation (\ref{Poisson-EQ}), which follows from Proposition~\ref{prop-h^{(g)}} below and \cite[Theorem~3.1]{Asmu94-QUESTA}. 

It is known \cite[Proposition 1.1]{Glyn96} that if 
$h$ is a solution of (\ref{Poisson-EQ}) and $\braket{\pi,  |h| + |h^{(g)}|} < \infty$ then, for any $c \in \bbR$,
\[
h(x) = h^{(g)}(x)  + c \quad \mbox{for $\pi$-almost everywhere $x \in \bbX$}.
\]
Therefore, the solutions of Poisson equation (\ref{Poisson-EQ}) are unique except the constant term if they are 
absolutely integrable solutions with respect to $\pi$. In addition, according to (\ref{defn-h^{(g)}}), $h^{(g)}$ is probabilistically interpretable and thus is tractable. From these reasons, we focus on the solution $h^{(g)}$ hereafter and, for convenience, we refer to it as the {\it standard solution} of Poisson equation (\ref{Poisson-EQ}).

Some researchers studied the standard solutions of the Poisson equations for structured Markov chains with countable state spaces. 
Dendievel et al.~\cite{Dend13}
derive computable results on the standard solution of the Poisson equation for quasi-birth-and-death processes (QBDs). Liu et al.~\cite{LiuYuan14} extend the results of \cite{Dend13} to GI/M/1-type Markov chains. Furthermore, Bini et al.~\cite{Bini16} discuss a general solution of the Poisson equation for QBDs. 

There are a few studies on the case of uncountable state spaces. Glynn~\cite{Glyn94} derive the standard solution of the Poisson equation for the waiting time sequence of the M/GI/1 queue. Asmussen and Baldt~\cite{Asmu94-QUESTA} extend Glynn~\cite{Glyn94}'s results to the workload process in a single-server queue with a Markovian arrival process (MAP; see \cite{Luca90}) and state-dependent service times, which is a generalization of the MAP/GI/1 queue considered in \cite{Luca90}. However, in general, the uncountability of state spaces leads to a difficulty in computing the standard solutions.

The main purpose of this paper is to present specific and tractable bounds for the solutions of Poisson equation (\ref{Poisson-EQ}) in the general setting. To this end, we assume the $f$-modulated drift condition (which is a kind of Foster-Lyapunov-type condition).
\begin{hm-cond}[$f$-modulated drift condition]\label{cond-drift}
For a given Borel function $f:\bbX \to (0,\infty)$ with $\inf_{x\in\bbX}f(x) > 0$, there exist some $b \in (0,\infty)$, {\it closed small set} $\bbC \subseteq
\bbX$ (see Remark~\ref{rem-small-set} below), and an extended-valued nonnegative function
$V$ on $\bbX$ satisfying $V(x_*) < \infty$ for some $x_* \in \bbX$, such that
\begin{equation}
\scA V \le  - f + b 1_{\bbC},
\label{ineqn-Qv-general}
\end{equation}
where, for any set $\bbA \subseteq \bbX$, a function
$1_{\bbA}: \bbX \to \{0,1\}$ is defined as
\[
1_{\bbA}(x)
=\left\{
\begin{array}{ll}
1, & \qquad x \in \bbA,
\\
0, & \qquad x \in \bbX \setminus \bbA.
\end{array}
\right.
\]
\end{hm-cond}

\begin{rem}\label{rem-small-set}
A set $\bbC \subseteq \bbX$ is said to be {\it small} if there exist some constant $T > 0$ and nonnegative measure $\nu$ on $(\bbX,\calB(\bbX))$ such that $\nu(\bbX) > 0$ and
\begin{equation}
P^T(x,\bbA) \ge \nu(A) \quad \mbox{for all $x \in \bbX$ and $A \in \calB(\bbX)$.}
\label{cond-small-set}
\end{equation}
Furthermore, suppose that (\ref{cond-small-set}) holds, and let $\mathfrak{m}$ denote a measure on $(\bbR_+,\calB(\bbR_+))$ such that $\calB(\{T\}) = 1$ and $\calB(\bbR_+) < \infty$. It then follows from (\ref{cond-small-set}) that
\[
\int_{t \in \bbR_+} \mathfrak{m}(dt) P^t(x,\bbA) 
\ge  P^T(x,\bbA) \ge \nu(A),\qquad \bbA \in \calB(\bbX).
\]
which shows that the small set $\bbC$ is an {\it $\mathfrak{m}$-petite set} (see, e.g., \cite[Section~4]{Meyn93-II}).
\end{rem}

\begin{rem}\label{rem-f-drift}
Suppose that $\{X(t)\}$ is non-explosive and $\psi$-irreducible.
If Condition~\ref{cond-drift} holds, then $\{X(t)\}$ is positive Harris recurrent and its invariant probability measure $\pi$ is unique (see \cite[Theorem~7]{Meyn93-proc}). Furthermore, $\pi$ satisfies
\begin{equation}
\pi \scA = 0,
\label{eqn-pi*A=0}
\end{equation}
which is proved in Appendix~\ref{appen-proof-eqn-pi*A=0}.
\end{rem}

Under Condition~\ref{cond-drift}, 
Glynn and Meyn~\cite{Glyn96} prove that Poisson equation (\ref{Poisson-EQ}) has a solution $h$ such that, for some $c_0 > 0$ and any $|g| \le f$, 
\begin{equation*}
|h| \le c_0 (V+1),
\end{equation*}
where the constant $c_0$ is not specified (see Theorem~3.2 therein).
Masuyama~\cite{Masu17-JORSJ} provides a procedure for computing such a constant, though the state space $\bbX$ is assumed to be countable.

In this paper, we derive specific bounds for the standard solution $h^{(g)}$ on the general space $\bbX$, though we need some additional conditions. We assume (see Condition~\ref{cond-petite} and Lemma~\ref{lem-xi_T} below) that for some $T > 0$ there exists a constant $\xi_T \in (0,1)$ such that 
\begin{eqnarray}
&&\inf_{x \in \bbC}P^T(x,\alpha) 
\ge \xi_T.
\label{cond-petite-01}
\end{eqnarray}
Under this condition, we show that
\begin{equation}
|h^{(g)}| 
\le \left(1 + { |\braket{\pi,  g}| \over \inf_{y \in \bbX} f(y)}  \right)
\left(
V_0 + { b T \over \xi_T}
\right)
\quad \mbox{for all $|g| \le f$,}
\label{bound-h-general}
\end{equation}
where 
\begin{equation}
V_0 = V - \inf_{y\in\alpha}V(y).
\label{defn-V_0}
\end{equation}

We now  note that the bound (\ref{bound-h-general}) requires $|g| \le f$, though this does not cause any restriction on its applicability. Indeed, it follows from (\ref{Poisson-EQ}) and (\ref{defn-h^{(g)}}) that
\begin{equation}
ch^{(g)} = h^{(cg)},\qquad c > 0,
\label{eqn-ch^{(g)}}
\end{equation}
that is, $ch^{(g)}$ is the standard solution of the following Poisson equation:
\[
\scA h = cg - \braket{\pi,  cg}.
\]
It also follows from (\ref{ineqn-Qv-general}) that
\[
\scA (cV) \le  - cf + cb 1_{\bbC}.
\]
Therefore, (\ref{bound-h-general}) implies that, for all $|g| \le f$ and $c>0$,
\begin{eqnarray*}
|h^{(cg)}| 
&\le& \left(1 + { |\braket{\pi,  cg}| \over c\inf_{y \in \bbX} f(y)}  \right)
\left(
cV_0 + { cb T \over \xi_T}
\right)
\nonumber
\\
&=& c\left(1 + { |\braket{\pi,  g}| \over \inf_{y \in \bbX} f(y)}  \right)
\left(
V_0 + { b T \over \xi_T}
\right).
\end{eqnarray*}
Combining this and (\ref{eqn-ch^{(g)}}) yields
\begin{eqnarray*}
|h^{(g)}| 
&\le& c\left(1 + { |\braket{\pi,  g}| \over c\inf_{y \in \bbX} f(y)}  \right)
\left(
V_0 + { b T \over \xi_T}
\right)\quad \mbox{for all $|g| \le cf$ and $c>0$}.~~
\end{eqnarray*}

Finally, we remark that if the small set $\bbC$ is finite then there exists a pair $(T,\xi_T)$ satisfying (\ref{cond-petite-01}) (which is proved in Lemma~\ref{lem-petite} below). Thus, we can readily find such a pair $(T,\xi_T)$ for specific Markov chains associated with familiar queueing models, such as M/GI/1 and MAP/GI/1 queues. Indeed, to demonstrate the applicability of our bounds, we apply them to the workload processes in two queues: a MAP/GI/1 queue; and an M/GI/1 queue with workload capacity limit (WCL), where the capacity can be infinite. For the first queue, we derive a computable bound for the standard solution to the Poisson equation of the workload process. For the second queue, we consider the workload processes of the finite and infinite models (the latter one is equivalent to an ordinary M/GI/1 queue), and establish an explicit bound for the difference between the stationary distributions of the two models.

The rest of this paper is divided into three sections. Section~\ref{sec-main-results} presents the main results of this paper. Sections~\ref{sec-MAP-GI-1} and \ref{sec-M-GI-1} applies them to the queueing examples.

\section{Main results}\label{sec-main-results}

This section presents the main results of this paper. We first introduce the formal definitions of the Markov chain $\{X(t);t\in\bbR_ \}$ and required notation together with technical conditions. We then present bounds for the standard solution $h^{(g)}$, given in (\ref{defn-h^{(g)}}), of Poisson equation (\ref{Poisson-EQ}). 

Let $\{X(t);t\in\bbR_+\}$ denote a continuous-time Markov chain on a Polish space $\bbX$ equipped with its Borel $\sigma$-field $\calB(\bbX)$. 
We then assume that $\{X(t)\}$ is a non-explosive Borel right process with the transition semigroup $\{P^t\}$ and thus it is strongly Markovian with right-continuous sample paths (see, e.g., \cite[pages 67--68 and Theorem~3.2.1]{Marc06}). We also assume that $\{X(t)\}$ is $\psi$-irreducible (see, e.g., \cite[Section 20.3.1]{Meyn09}); that it, the $\psi$-irreducibility of $\{X(t)\}$ is equivalent to
\[
\psi(\bbA) > 0 ~\Longrightarrow~ 
\EE_x \! \left[ \int_0^{\infty} \one(X(t) \in \bbA) dt \right] > 0
\quad \mbox{for all $x \in \bbX$,}
\]
where $\EE_x[\,\cdot\,]=\EE[\,\cdot\, |\, X(0)=x]$ and $\one(\,\cdot\,)$ denotes the indicator function. 

Let $\scB$ denote a Banach space that consists of real Borel functions $F$'s on $\bbX$ such that $\int_{x\in\bbX} |F(x)| \varphi(dx) < \infty$ for some probability measure $\varphi$ on $\calB(\bbX)$. 
Let $\calD$ denote the set of functions $V$'s in $\scB$ such that, for each $V \in \scB$, there exists a Borel function $U:\bbX\to\bbR$ that satisfies the following (see \cite[Section 1.3]{Meyn93-III}): For any  initial condition on $X(0)$,
\begin{equation*}
M(t) := V(X(t)) - \int_0^t U(X(u)) du,\qquad t \in \bbR_+,
\end{equation*}
is a {\it local martingale} (see, e.g., \cite{Kont16} and \cite[Section 26]{Davi93}). We then write $\scA V = U$ and refer to the operator $\scA$ as 
the {\it extended generator} of the $\psi$-irreducible Markov chain $\{X(t)\}$.

\begin{rem}
According to the definition of $\scA$, there exists an increasing sequence of stopping times, $\{s_m; m\in\bbZ_+\}$, such that $\lim_{m\to\infty}s_m = \infty$ with probability one and, for $t \in \bbR_+$ and $ m \in \bbZ_+$,
\begin{eqnarray}
\EE_x[V(X(t \vmin s_m))] 
= V(x) + \EE_x \! \left[ \int_0^{t \vmin s_m} \scA V(X(u)) du \right],
\qquad \forall V \in \calD.
\label{defn-extended-generator}
\end{eqnarray}
where $x \vmin y = \min(x,y)$ for $x,y\in\bbR$.
\end{rem}

\begin{rem}\label{rem-weak-generator}
Let $\wt{\scA}$ denote a linear operator such that
\begin{equation}
\wt{\scA} V(x) = \lim_{t \downarrow 0} {P^t V(x) - V(x) \over t},
\qquad x \in \bbX,
\label{defn-wt{A}}
\end{equation}
on
\[
\wt{\calD} = \{V \in \scB: \mbox{the limit in (\ref{defn-wt{A}}) exists for each $x \in \bbX$}\}.
\]
The operator $\wt{\scA}$ is referred to the {\it weak generator} of $\{X(t)\}$ (see, e.g., \cite[Chapter 1, Section 6]{Dynk65}). It follows from Fubini's theorem and Dynkin's formula (see \cite[Proposition 1.5]{Ethi05} and \cite[Equation (8)]{Meyn93-III}) that, for $t \in \bbR_+$,
\[
\EE_x[V(X(t))] 
= V(x) + \EE_x \! \left[ \int_0^{t} \wt{\scA} V(X(u)) du \right],
\qquad \forall V \in \wt{\calD},
\]
Therefore, the optional sampling (stopping) theorem (see, e.g., \cite[Section 5.3]{Lawl06}) yields, for $t \in \bbR_+$ and $ m \in \bbZ_+$,
\begin{eqnarray*}
\EE_x[V(X(t \vmin s_m))] 
= V(x) + \EE_x \! \left[ \int_0^{t \vmin s_m} \wt{\scA} V(X(u)) du \right],
\qquad \forall V \in \wt{\calD}.
\end{eqnarray*}
This equation together with (\ref{defn-extended-generator}) implies that $\wt{\calD} \subset \calD$ and 
\begin{equation}
\wt{\scA}V = \scA V,\qquad \forall V \in \wt{\calD}.
\label{eqn-wt{A}=A}
\end{equation}
\end{rem}

\medskip

We now make the following condition, which is necessary for the definition of the standard solution $h^{(g)}$.
\begin{hm-cond}\label{cond-atom}
There exists a set $\alpha \in \calB(\bbX)$ such that $\psi(\alpha) > 0$ and, for all $t>0$,
\begin{equation*}
P^t(x,\bbA) = \nu_t(\bbA),
\qquad x \in \alpha,\ \bbA\in\calB(\bbX),
\end{equation*}
where, for each $t > 0$, $\nu_t$ is a probability measure on $\calB(\bbX)$. The set $\alpha$ is referred to as an {\it (accessible) atom} (see, e.g., \cite[Chapter 5]{Meyn09}).
\end{hm-cond}

\begin{prop}\label{prop-h^{(g)}}
If Conditions~\ref{cond-drift} and \ref{cond-atom} hold, then 
(i) the Markov chain $\{X(t)\}$ is a regenerative process (see, e.g., \cite[Chapter VI]{Asmu03}) such that the return times to atom $\alpha$ are regeneration points; and (ii) the function $h^{(g)}$, given in (\ref{defn-h^{(g)}}), satisfies
\begin{eqnarray}
h^{(g)}(x) &=& 0,\qquad x \in \alpha,
\label{eqn-h^{(g)}(alpha)=0}
\\
h^{(g)}(x) 
&=& \EE_x \! \left[ \int_0^{\wt{\tau}_{\alpha}} g(X(t)) dt\right]
- \braket{\pi,  g} \EE_x[\wt{\tau}_{\alpha}],\qquad x \in \bbX,
\label{eqn-h^{(g)}(x)}
\end{eqnarray}
where $\wt{\tau}_{\alpha}:=\inf\{t \in \bbR_+: X(t) \in \alpha\}$ is the first hitting time to atom $\alpha$.
\end{prop}

\begin{proof}
The statement (i) follows from the strong Markov property and the definition of atom $\alpha$. It also follows from the first equation at page 244 of \cite{Asmu94-QUESTA} that
\[
\braket{\pi,  g} 
= {1\over \EE_x[\tau_{\alpha}]}
\EE_x \! \left[ \int_0^{\tau_{\alpha}} g(X(t)) dt\right],
\qquad x \in \alpha.
\]
Combining this and (\ref{defn-h^{(g)}}) leads to (\ref{eqn-h^{(g)}(alpha)=0}).  

It remains to prove (\ref{eqn-h^{(g)}(x)}).
By definition, $\wt{\tau}_{\alpha} = 0$ if $X(0) \in \alpha$. Therefore, for $x \in \alpha$,
\[
\EE_{x} \! \left[ \int_0^{\wt{\tau}_{\alpha}} g(X(t)) dt\right]
- \braket{\pi,  g} \EE_{x}[\wt{\tau}_{\alpha}]
= 0 = h^{(g)}(x).
\]
where the second equality is due to (\ref{eqn-h^{(g)}(alpha)=0}).
Furthermore, if $X(0) \not \in \alpha$ then $\tau_{\alpha} = \wt{\tau}_{\alpha}$ and thus, for $x \not\in \alpha$,
\[
h^{(g)}(x)
=
\EE_x \! \left[ \int_0^{\wt{\tau}_{\alpha}} g(X(t)) dt\right]
- \braket{\pi,  g} \EE_x[\wt{\tau}_{\alpha}].
\]
As a result, (\ref{eqn-h^{(g)}(x)}) holds for all $x \in \bbX$. \qed
\end{proof}

Proposition~\ref{prop-h^{(g)}} together with \cite[Theorem~3.1]{Asmu94-QUESTA}
implies that $h^{(g)}$, given in (\ref{defn-h^{(g)}}), is a solution of Poisson equation (\ref{Poisson-EQ}). Namely, the standard solution $h^{(g)}$ is well-defined. 

\begin{rem}
We can define the standard solution $h^{(g)}$ as in (\ref{eqn-h^{(g)}(x)}). Indeed, this alternative definition is adopted in \cite{LLM2018,Masu17-JORSJ}.
\end{rem}

To proceed further, we require Condition~\ref{cond-petite} below.
\begin{hm-cond}\label{cond-petite}
For some $T_* \in \bbR_+$, there exists a constant $\xi_{T_*} \in (0,1]$ such that
\[
\inf_{x\in\bbC}P^{T_*}(x,\alpha) \ge \xi_{T_*},
\]
where $\bbC$ is the closed small set that appears in Condition~\ref{cond-drift}.
\end{hm-cond}

\begin{rem}
Condition~\ref{cond-petite} is satisfied if the small set $\bbC$ is finite (see Lemma~\ref{lem-petite}).
\end{rem}

Under Conditions~\ref{cond-drift}, \ref{cond-atom} and \ref{cond-petite}, we show a lemma used to derive bounds for $|h^{(g)}|$.
\begin{lem}\label{lem-xi_T}
If Conditions~\ref{cond-drift}, \ref{cond-atom} and \ref{cond-petite} are satisfied, then, for each $T \ge T_*$ there exists some constant $\xi_T \in (0,1]$ such that
\begin{equation}
\inf_{x\in\bbC}P^T(x,\alpha) \ge \xi_T.
\label{cond-petite-02}
\end{equation}
\end{lem}

\begin{proof}
Since $\alpha$ is an atom, there exists some $c_* > 0$ such that
\begin{equation}
\PP(X(u) \in \alpha, \forall u \in [0,t] \mid X(0) \in \alpha)
= e^{-c_* t}\qquad \mbox{for all $t \in \bbR_+$.}
\label{indeqn-atom}
\end{equation}
Using (\ref{indeqn-atom}) and Condition~\ref{cond-petite}, we have, for all $t \in \bbR_+$,
\begin{eqnarray*}
\inf_{x\in\bbX} P^{t+T_*}(x,\alpha)
&\ge& \inf_{x\in\bbX} P^{T_*}(x,\alpha) P^t(\alpha,\alpha)
\nonumber
\\
&\ge& \xi_{T_*} P^t(\alpha,\alpha)
\nonumber
\\
&\ge& \xi_{T_*} \PP(X(u) \in \alpha, \forall u \in [0,t] \mid X(0) \in \alpha)
\nonumber
\\
&\ge& \xi_{T_*} e^{-c_* t},
\end{eqnarray*}
which completes the proof. \qed
\end{proof}

We are now ready to present the bound (\ref{bound-h-general}) for $|h^{(g)}|$.
\begin{thm}\label{thm-bound-h}
Suppose that Conditions~\ref{cond-drift}, \ref{cond-atom} and \ref{cond-petite} are satisfied. For $T\ge T_*$, let $\xi_T$ be a constant such that (\ref{cond-petite-02}) holds. We then have the bound (\ref{bound-h-general}), more precisely,
\begin{equation}
|h^{(g)}| 
\le \left(1 + { |\braket{\pi,  g}| \over \inf_{y \in \bbX} f(y)}  \right)
\left(
V_0 + { b T \over \xi_T}
\right)
\quad \mbox{for all $|g| \le f$ and $T \ge T_*$.}
\label{bound-h-general-stronger}
\end{equation}
We also have a weaker bound insensitive to $g$:
\begin{equation}
|h^{(g)}| 
\le \left(1 + { b\pi(\bbC) \over \inf_{y \in \bbX} f(y)}  \right)
\left(
V_0 + { b T \over \xi_T }
\right)
\quad \mbox{for all $|g| \le f$ and $T \ge T_*$.}
\label{bound-h-general-weaker}
\end{equation}
\end{thm}

\begin{proof}
See Appendix~\ref{proof-main-thm-general}. \qed
\end{proof}

\begin{rem}
Since $\pi(\bbC) \le 1$, the bound (\ref{bound-h-general-weaker}) yields
\[
|h^{(g)}| 
\le \left(1 + { b \over \inf_{y \in \bbX} f(y)}  \right)
\left(
V_0 + { b T \over \xi_T }
\right)
\quad \mbox{for all $|g| \le f$ and $T \ge T_*$.}
\]
\end{rem}

\medskip

When $\bbC = \alpha$, we have the following result.
\begin{coro}\label{coro-explicit-bound}
If Conditions~\ref{cond-drift}, \ref{cond-atom} and \ref{cond-petite} hold with $\bbC = \alpha$, then
\begin{eqnarray}
|h^{(g)}| 
&\le& \left(1 + {|\braket{\pi,  g}| \over \inf_{y\in \bbX}f(y)} \right) V_0
\le   \left(1 + {b\pi(\alpha) \over \inf_{y\in \bbX}f(y)} \right) V_0
\quad\mbox{for all $|g| \le f$}. \quad
\label{explicit-bounds-h-general}
\end{eqnarray}
\end{coro}

\begin{proof}
From (\ref{indeqn-atom}), we have
\[
P^T(\alpha,\alpha) \ge e^{-c_* T}\quad \mbox{for all $T \in \bbR_+$.}
\]
Thus, Theorem~\ref{thm-bound-h} yields the bounds (\ref{bound-h-general-stronger}) and (\ref{bound-h-general-weaker}) with $\xi_T = e^{-c_* T}$. Letting $T \downarrow 0$ in these bounds, we obtain (\ref{explicit-bounds-h-general}). \qed
\end{proof}

\section{Application to a MAP/GI/1 queue}\label{sec-MAP-GI-1}

This section discusses the application of Theorem~\ref{thm-bound-h} to a MAP/GI/1 queue.
The system has a single server and a waiting room of infinite capacity. The arrivals of customers form a Markovian arrival process (MAP) \cite{Luca90}, which is controlled by an irreducible Markov chain $\{J(t);t \in \bbR_+\}$ with a finite state space $\bbM:=\{1,2,\dots,M\}$. Let $N(t)$, $t \in \bbR_+$, denote the total number of arrivals in the interval $(0,t]$. We assume that $N(0)=0$ and, for $i,j\in\bbM$,
\begin{eqnarray*}
&&
\PP(N(t+\Delta t) - N(t) = k, J(t+\Delta t) = j \mid J(t) = i)
\nonumber
\\
&&{} \quad
= \left\{
\begin{array}{ll}
\delta_{i,j} + C_{i,j} \Delta t + o(\Delta t), 	& \quad k=0,
\\
D_{i,j} \Delta t + o(\Delta t), 				& \quad k=1,
\\
o(\Delta t), 									& \quad k=2,3,\dots,
\end{array}
\right.
\end{eqnarray*}
where $\delta_{i,j}$ denotes the {\it Kronecker delta}, and where $o(t)$ represents some function such that, if divided by $t$, it converges to zero as $t \to 0$. Let $\vc{C}= (C_{i,j})_{i,j\in\bbM}$ and $\vc{D}= (D_{i,j})_{i,j\in\bbM}$. It then follows that $\vc{C} + \vc{D}$ is the infinitesimal generator of the irreducible  Markov chain $\{J(t)\}$ and thus has a unique stationary probability vector, denoted by $\vc{\varpi}:=(\varpi_i)_{i\in\bbM}$. We now define $\lambda = \vc{\varpi}\vc{D}\vc{e}$, where $\vc{e}=(1,1,\dots,1)^{\top}$. The factor $\lambda$ is called the arrival rate. 

As described above, customers arrive according to MAP characterized by a pair $(\vc{C},\vc{D})$. We assume that arriving customers are served on a first-come-first-served basis and their service times are independent and identically distributed (i.i.d.) with a distribution $H$ such that 
\[
H(0) = 0,\quad \mu^{-1}:=\int_0^{\infty} x H(dx) \in (0,\infty).
\]
This queue is referred to as MAP/GI/1 queue. 

Let $W(t)$, $t \in \bbR_+$, denote the workload (i.e., the total unfinished work) in the system at time $t$. 
Assume that $\rho := \lambda /\mu \in (0,1)$. 
Thus, $\{X(t):=(W(t),J(t));t \in\bbR_+\}$ is an ergodic Markov chain with state space $\bbX:=\{(x,i) \in \bbR_+ \times \bbM\}$  (see, e.g., \cite{Loyn62}). Let $\{P^t;t\in \bbR_+\}$ be the transition semigroup of the Markov chain $\{X(t)\}$. Moreover, let $V:\bbX\to\bbR_+$ be a function such that,  for any fixed $i \in \bbM$, $V(x,i)$ is differentiable with respect to $x \in \bbR_+$, and let $\vc{v}(x)=(V(x,i))_{i\in\bbM}$ for $x \in \bbR_+$. It then follows that, for $x \in \bbR_+$,
\begin{eqnarray}
P^{t}\vc{v}(x) 
&=& \left( \int_{y\in\bbR_+} \sum_{j\in\bbM}P^t((x,i),(dy,j)) V(y,j) \right)_{\!\! i\in\bbM}
\nonumber
\\
&=&
(\vc{I} + \vc{C} t) \vc{v}( (x - t)^+) 
+ \int_0^{\infty} \vc{D} t H(dy) \vc{v}(x+y)  + o(t),
\label{eqn-P^tV(x)}
\end{eqnarray}
where $\vc{I} $ denotes the identity matrix and  $(x)^+ = \max(x,0)$ for $x \in \bbR_+$. It also follows from (\ref{eqn-P^tV(x)}) (see (\ref{eqn-wt{A}=A}) in Remark~\ref{rem-weak-generator}) that
\begin{align}
\scA\vc{v}(x)
&=\lim_{t\to0}{P^{t}\vc{v}(x) - \vc{v}(x) \over t}
\nonumber
\\
&=
\left\{
\begin{array}{ll}
\vc{C} \vc{v}(0) + \dm\int_0^{\infty} \vc{D} \vc{v}(y)H(dy), 
& \quad x= 0,
\\
- \vc{v}'(x) + \vc{C} \vc{v}(x) + \dm\int_0^{\infty} \vc{D}\vc{v}(x+y)H(dy), & \quad x > 0,
\end{array}
\right.
\label{defn-A-BMAP-GI-1}
\end{align}
where $\vc{v}'(x) = (V'(x,i))_{i\in\bbM}$ for $x \in \bbR_+$.

We assume that $H$ is light-tailed, i.e., 
\begin{equation}
\ol{\theta}:=\sup \left\{\theta \in \bbR_+: 
 \int_0^{\infty} e^{\theta x} H(dx) < \infty \right\} > 0.
\label{cond-light-tailed}
\end{equation}
Let $\sigma(\theta)$, $\theta \in (-\infty,\ol{\theta})$, denote a real maximum eigenvalue of $\vc{C} + \wh{H}(\theta)\vc{D}$, where
\[
\wh{H}(\theta)=\int_0^{\infty} e^{\theta x} H(dx),
\qquad \theta < \ol{\theta}.
\]
There exists some $K:=K(\theta) > 0$ such that $\vc{C} + \wh{H}(\theta)\vc{D} + K\vc{I} \ge \vc{O}$ is irreducible and thus it has a positive right eigenvector, denoted by $\vc{u}(\theta):=(u(\theta,i))_{i\in\bbM} > \vc{0}$, belonging to Perron-Frobenius eigenvalue $\sigma(\theta)+K$ (see, e.g., \cite[Theorem 8.4.4]{Horn13}). Therefore, for $\theta \in (-\infty,\ol{\theta})$, $\sigma(\theta)$ is a simple eigenvalue of $\vc{C} + \wh{H}(\theta)\vc{D}$ and 
\begin{equation}
\{\vc{C} + \wh{H}(\theta) \vc{D}\}\vc{u}(\theta)
= \sigma(\theta)\vc{u}(\theta).
\label{eqn-C+D(theta)u(theta)}
\end{equation}

Clearly, $\vc{C} + \wh{H}(\theta)\vc{D}$ is differentiable (with respect to $\theta$). Thus, we can assume that $\vc{u}(\theta)$ is differentiable (see \cite[Chapter 9, Theorem 8]{Lax07}). Furthermore, $\sigma(\theta)$ is differentiable (see \cite[Chapter 9, Theorem 7]{Lax07}). Note here that $\vc{\varpi}(\vc{C}+\vc{D}) = \vc{0}$, $\sigma(0)=0$ and $\vc{u}(0) = c\vc{e}$ for some $c > 0$. Using these facts, we calculate $\sigma'(0)$ from (\ref{eqn-C+D(theta)u(theta)}), which results in
\[
\sigma'(0) = \vc{\varpi} \vc{D} \vc{e} 
\left. {d \over d\theta}\wh{H}(\theta) \right|_{\theta=0}
= \lambda /\mu = \rho < 1.
\]
Therefore, $\sigma(\theta) < \theta$ for some $\theta > 0$. 

In what follows, we fix $\theta$ such that $\theta > 0$ and $\sigma(\theta) < \theta$. We also assume, without loss of generality, that
\begin{equation}
\max_{j\in\bbM} u(\theta,j) = 1.
\label{eqn-max-u(theta,j)}
\end{equation}
We then fix $\vc{v}(x) = (V(x,i))_{i\in\bbM}$ such that
\begin{eqnarray}
\vc{v}(x) = e^{\theta x} \vc{u}(\theta),
\qquad x \in \bbR_+.
\label{eqn-v(x)-BMAP-GI-1}
\end{eqnarray}
Substituting (\ref{eqn-v(x)-BMAP-GI-1}) into (\ref{defn-A-BMAP-GI-1}), and using (\ref{eqn-C+D(theta)u(theta)}) yields
\begin{align*}
\scA\vc{v}(0)
&=
\{\vc{C} + \wh{H}(\theta) \vc{D}\} \vc{u}(\theta)
= \sigma(\theta) \vc{u}(\theta),
\end{align*}
and
\begin{align*}
\scA\vc{v}(x)
&=
\left[
- \theta\vc{I} + \{\vc{C} + \wh{H}(\theta) \vc{D}\}
\right] e^{\theta x}  \vc{u}(\theta)
\nonumber
\\
&= -(\theta - \sigma(\theta)) e^{\theta x} \vc{u}(\theta)
\nonumber
\\
&= -(\theta - \sigma(\theta)) \vc{v}(x) < 0, \qquad x > 0.
\end{align*}
These equations together with (\ref{eqn-max-u(theta,j)}) lead to
\begin{equation*}
\scA V \le -(\theta - \sigma(\theta)) V 
+ \theta 1_{\{0\} \times \bbM}.
\end{equation*}
Therefore, Condition~\ref{cond-drift} holds with
\begin{eqnarray}
b = \theta,
\quad
f = (\theta - \sigma(\theta)) V,
\quad
\bbC = \{0\} \times \bbM,
\label{eqn-b-f-C}
\end{eqnarray}
where $V$ is given in (\ref{eqn-v(x)-BMAP-GI-1}).

Let $i_0 \in \argmax_{j\in\bbM} u(\theta,j)$. Equation (\ref{eqn-max-u(theta,j)}) then lead to $u(\theta,i_0) = 1$. Thus, (\ref{defn-V_0}) and (\ref{eqn-v(x)-BMAP-GI-1}) yield
\begin{equation}
V_0(x,i) = u(\theta,i)e^{\theta x} - 1,
\qquad (x,i) \in \bbR_+ \times \bbM.
\label{eqn-V_0-MAP-GI-1}
\end{equation}
Note here that $\alpha:=(0,i_0) \in \bbX$ is an atom, which shows that Condition~\ref{cond-atom} holds. 

We now fix $x_0 > 0$ such that $H(x_0) > 0$, and recall that $\vc{C} + \vc{D}$ is an irreducible generator of the Markov chain $\{J(t)\}$ with state space $\bbM=\{1,2,\dots,M\}$. It then follows that, for any $t_0 > 0$ and $i\in\bbM$,
\begin{eqnarray*}
P^{t_0+Mx_0}((0,i),(0,i_0)) 
\ge \left[
\exp\{\vc{C} t_0\}
\left(
H(x_0)\vc{D} \exp\{\vc{C} x_0\}
\right)^M
\right]_{i,i_0} > 0,
\end{eqnarray*}
where $[\,\cdot\,]_{i,j}$ denotes the $(i,j)$-th element of the matrix in the square brackets. 
Therefore, Condition~\ref{cond-petite} holds with 
\begin{align}
\alpha & =(0,i_0),
\nonumber
\\
T & = t_0+Mx_0,
\label{eqn-alpha-T}
\\
\xi_T &= [ H(x_0) ]^M 
\min_{i\in\bbM}\left[
\exp\{\vc{C} t_0\}
\left(
\vc{D} \exp\{\vc{C} x_0\}
\right)^M
\right]_{i,i_0}.
\label{defn-xi_T-MAP-GI-1-01}
\end{align}

We have confirmed that the conditions of Theorem~\ref{thm-bound-h} are satisfied in the present setting. Note here that (by Little's law) 
\[
\pi(\bbC) = \pi(\{0\} \times \bbM) = 1 - \rho.
\]
It thus follows from Theorem~\ref{thm-bound-h}, (\ref{eqn-b-f-C}) and (\ref{eqn-V_0-MAP-GI-1}) that, for all $(x,i) \in \bbR_+ \times \bbM$ and $|g| \le (\theta - \sigma(\theta))V$,
\begin{eqnarray}
|h^{(g)}(x,i)|
&\le& 
\left(
1 + { \theta (1- \rho) \over \{\theta - \sigma(\theta)\} \dm\min_{j\in\bbM}u(\theta,j) }
\right)
\left\{
 u(\theta,i) e^{\theta x} - 1
+  {\theta T \over \xi_T}  
\right\}, \qquad~
\label{bound-MAP-GI-1-01}
\end{eqnarray}
where $T$ and $\xi_T$ are given in (\ref{eqn-alpha-T}) and (\ref{defn-xi_T-MAP-GI-1-01}), respectively.

The bound (\ref{bound-MAP-GI-1-01}) includes $\sigma(\theta)$, $\vc{u}(\theta)$, and $\exp\{\vc{C} t\}$ ($t=t_0,x_0$). The Perron-Frobenius eigenvalue $\sigma(\theta)$ and vector $\vc{u}(\theta)$ can be computed by a common method, such as the {\it power method}. The matrix exponential $\exp\{\vc{C} t\}$ can be computed by the uniformization technique (see, e.g., \cite[Section~4.5.2]{Tijm03}):
\[
\exp\{\vc{C} t\} 
= \sum_{\ell=0}^{\infty} e^{-\zeta t} {(\zeta t)^{\ell} \over \ell !}
[\vc{I} + \zeta^{-1} \vc{C}]^{\ell},\qquad t \in \bbR_+,
\]
where $\zeta = \max_{i\in\bbM} |C_{i,i}|$. Therefore, the bound (\ref{bound-MAP-GI-1-01}) can be computable, provided that $T / \xi_T$ is given. However, we cannot readily obtain an explicit expression of $\xi_T$ in the general setting.

We now consider a special case. Suppose that 
\begin{equation}
C_{i,i_0} > 0\quad \mbox{for all $i \in \bbM \setminus \{i_0\}$},
\label{cond-C_{i,i_0} > 0}
\end{equation}
and fix
\begin{equation}
\xi_T = \min_{i\in\bbM}[\exp\{\vc{C} T\}]_{i,i_0} > 0,
\qquad T  > 0,
\label{defn-xi_T-MAP-GI-1-02}
\end{equation}
which leads to
\[
P^T((0,i),(0,i_0)) \ge \xi_T,\qquad T  > 0.
\]
Therefore, substituting (\ref{defn-xi_T-MAP-GI-1-02}) into (\ref{bound-MAP-GI-1-01}) yields
\begin{eqnarray*}
|h^{(g)}(x,i)|
&\le& 
\left(
1 + { \theta (1- \rho) \over \{\theta - \sigma(\theta)\} \dm\min_{j\in\bbM}u(\theta,j) }
\right)
\nonumber
\\
&& {} \times 
\left\{
 u(\theta,i) e^{\theta x} - 1
+  {\theta T \over \dm\min_{i\in\bbM}[\exp\{\vc{C} T\}]_{i,i_0}}  
\right\}, \qquad T > 0.
\end{eqnarray*}
Letting $T \downarrow 0$ in this inequality, we obtain, for $(x,i) \in \bbR_+ \times \bbM$,
\begin{eqnarray}
|h^{(g)}(x,i)|
&\le& 
\left(
1 + { \theta (1- \rho) \over \{\theta - \sigma(\theta)\} \dm\min_{j\in\bbM}u(\theta,j) }
\right)
\left\{
 u(\theta,i) e^{\theta x} - 1
+  {\theta \over \dm\min_{i\in\bbM}C_{i,i_0}}  
\right\}. \qquad
\label{bound-MAP-GI-1-02}
\end{eqnarray}

Compared with (\ref{bound-MAP-GI-1-01}), this bound (\ref{bound-MAP-GI-1-02}) replaces the {\it troublesome} factor $T / \xi_T$ by $1/\min_{i\in\bbM}C_{i,i_0}$ under the additional condition (\ref{cond-C_{i,i_0} > 0}). Recall here that $T / \xi_T$ vanishes if $\bbC=\{0\} \times \bbM$ is an atom (see Corollary~\ref{coro-explicit-bound}). Such a {\it favorable} queueing model is considered in the next section.

\section{Application to an M/GI/1--WCL queue}\label{sec-M-GI-1}

\subsection{Model description and basic results}

This section considers an M/GI/1 queue with workload capacity limit (WCL) $L \in (0,\infty]$ \cite{Perr01}. Customers arrive at the system according to a Poisson process with rate $\lambda \in (0,\infty)$, and their service times are positive (with probability one) and i.i.d.\ with distribution $H$ having mean $\mu^{-1} \in (0,\infty)$.

An arriving customer is accepted if the total workload including its service time is not greater than the limit $L$; otherwise the customer is rejected. We refer to this queueing model as the M/GI/1--WCL queue. Note that if $L = \infty$ then the M/GI/1--WCL queue is reduced to an ordinary M/GI/1 queue, which accepts all arriving customers.

We first consider the finite model, i.e., the case of $L < \infty$.
Let $X_L(t)$, $t \in \bbR_+$, denote the workload in the finite model at time $t$. The workload process $\{X_L(t);t \in\bbR_+\}$ is a positive Harris chain with state space $\bbX = [0,L]$ and that its transition semigroup $\{P_L^t;t\in\bbR_+\}$ satisfies the following:
\begin{eqnarray}
P_L^{t}V(x) 
&=& \int_0^L P_L^t(x,dy) V(y) 
\nonumber
\\
&=& [1 - \lambda t H(L-x) ] V( (x - t)^+) 
\nonumber
\\
&& {} \quad
+ \int_0^{L-x} \lambda t H(dy) V(x+y)  + o(t), 
\qquad x \in [0,L],
\label{eqn-P^tV(x)-MG1}
\end{eqnarray}
where $V:\bbR_+ \to \bbR_+$ (which appears hereafter in this section) denotes a differentiable function. 
Furthermore, let $\scA_L$ denote the extended generator of $\{X_L(t)\}$  (see (\ref{eqn-wt{A}=A}) in Remark~\ref{rem-weak-generator}). It then follows from (\ref{eqn-P^tV(x)-MG1}) that, for $0 \le x \le L$,
\begin{eqnarray}
\scA_L V(x)
&=& - 1_{(0,\infty)}(x) V'(x) -\lambda H(L-x) V(x) 
+ \lambda\dm\int_0^{L-x} H(dy) V(x+y)
\nonumber
\\
&=& - 1_{(0,\infty)}(x) V'(x) 
+ \lambda\dm\int_0^{L-x} H(dy) \{ V(x+y) - V(x) \}. 
\label{defn-A-MG1-finite-02}
\end{eqnarray}
Since $\{X_L(t)\}$ is positive Harris, it has a unique invariant probability measure, denoted by $\pi_L$, on $\calB([0,L])$. The invariant probability measure $\pi_L$ satisfies the equilibrium equation:
\begin{eqnarray}
0= \pi_L \scA V
&=& \int_0^L \pi(dx) 
\bigg[
- 1_{(0,\infty)}(x) V'(x) 
\nonumber
\\
&& {} \qquad \qquad \qquad
\lambda\dm\int_0^{L-x} H(dy) \{ V(x+y) - V(x) \}
\bigg].
\label{eqn-pi_L-A_L-V=0}
\end{eqnarray}

Next we consider the infinite model, i.e., the case of $L = \infty$.
Let $X(t)$, $t \in \bbR_+$, denote the workload in the infinite model at time $t$. Let $\{P^t;t\in\bbR_+\}$ denote the transition semigroup of the Markov chain $\{X(t);t\in\bbR_+\}$. We then have
\begin{eqnarray}
P^{t}V(x) 
&=& (1 - \lambda t) V( (x - t)^+) 
+ \int_0^{\infty} \lambda t H(dy) V(x+y)  + o(t), 
\quad x \in \bbR_+. \qquad
\label{eqn-P^tV(x)-MG1-infinite}
\end{eqnarray}
Therefore, the extended generator of $\{X(t)\}$, denoted by $\scA$, satisfies the following (see Remark~\ref{rem-weak-generator}):
\begin{eqnarray}
\scA V(x)
&=& - 1_{(0,\infty)}(x) V'(x) 
+ \lambda\dm\int_0^{\infty} H(dy) \{ V(x+y) - V(x) \},  \quad x\in\bbR_+.\qquad
\label{defn-A-MG1}
\end{eqnarray}

\begin{rem}\label{rem-A-MG1}
Suppose that
\begin{equation}
\lim_{y\to\infty}\ol{H}(y)V(x+y) = 0\quad \mbox{for any fixed $x \in \bbR_+$}.
\label{cond-H}
\end{equation}
It then follows from (\ref{defn-A-MG1}) that
\begin{eqnarray}
\scA V(x)
&=& - 1_{(0,\infty)}(x) V'(x) 
+ \lambda \dm\int_0^{\infty} \ol{H}(y) V'(x+y)dy,  \quad x\in\bbR_+.\qquad
\label{defn-A-MG1-02}
\end{eqnarray}
\end{rem}

In what follows, we assume that $\rho = \lambda/\mu \in (0,1)$, under which $\{X(t)\}$ is positive Harris recurrent with a unique invariant probability measure, denoted by $\pi$, on $\calB(\bbR_+)$. It is known (see, e.g., \cite[Section 5.1.5]{Gros08}) that
\begin{equation}
\pi(dx) = (1 - \rho)\sum_{n=0}^{\infty} H_{\re}^{*n}(dx),
\label{eqn-pi(x)-MG1}
\end{equation}
where $H_{\re}$ is the equilibrium distribution of $H$ and $H_{\re}^{*n}$ is the $n$-fold convolution of itself, i.e.,
\begin{align*}
H_{\re}^{*0}(x) &= 1_{\bbR_+}(x), & x & \in \bbR_+,
\\
H_{\re}^{*1}(x) &= H_{\re}(x) = \mu^{-1}\int_0^x \ol{H}(y) dy, & x & \in \bbR_+,
\\
H_{\re}^{*n}(x) &= \int_0^x H_{\re}^{*(n-1)}(x-y) H_{\re}(dy), & n\ge2,~ x & \in \bbR_+,
\end{align*}
where $\ol{H}= 1 - H$.

\subsection{A bound for the distance between the stationary distributions of the finite and infinite models}

In this subsection, we consider a distance between the stationary distributions $\pi_L$ and $\pi$. To this end, we extend the finite chain $\{X_L(t)\}$ on $[0,L]$ to the infinite space $\bbR_+$, and then modify its transition semigroup $\{P_L^t\}$ in such a way that
\begin{eqnarray}
P_L^{t}V(x) 
&=& [1 - \lambda t H(L-x) ] V( (x - t)^+) 
\nonumber
\\
&& {} \quad
+ \int_0^{L-x} \lambda t H(dy) V(x+y)  + o(t), 
\qquad 0 \le x \le L,
\label{eqn-P^tV(x)-MG1-a}
\\
P_L^{t}V(x) 
&=& (1 - \lambda t) V( (x - t)^+) 
\nonumber
\\
&& {} \quad
+ \int_0^{\infty} \lambda t H(dy) V(x+y)  + o(t), 
\qquad ~~\,x > L.
\label{eqn-P^tV(x)-MG1-b}
\end{eqnarray}

For this modified chain $\{X_L(t)\}$, we denote by $\pi_L$ and $\scA_L$, its invariant probability measure and extended generator, respectively. Note that (\ref{eqn-P^tV(x)-MG1-a}) is the same as (\ref{eqn-P^tV(x)-MG1}) and thus (\ref{defn-A-MG1-finite-02}) still holds for $0 \le x \le L$. Furthermore, (\ref{eqn-P^tV(x)-MG1-infinite}) and (\ref{eqn-P^tV(x)-MG1-b}) show that the modified chain $\{X_L(t)\}$ evolves in the same way as the infinite chain $\{X(t)\}$ while the former is in $(L,\infty)$. Therefore, we have
\begin{eqnarray}
\lefteqn{
\scA_L V(x)
}
\quad &&
\nonumber
\\
&=&
\left\{
\begin{array}{l@{~~~}l}
- 1_{(0,\infty)}(x) V'(x) 
+ \lambda\dm\int_0^{L-x} H(dy) \{V(x+y)-V(x)\},     &  0 \le x \le L,
\\
\scA V(x), & x > L,
\end{array}
\right. \qquad
\label{defn-A-MG1-finite-03}
\end{eqnarray}
where the generator $\scA$ is specified by (\ref{defn-A-MG1}). In addition, (\ref{eqn-P^tV(x)-MG1-a}) implies that $\{X_L(t)\}$ never reaches from $[0,L]$ to any state in $(0,L)$ and thus
\begin{equation*}
\pi_L((L,\infty))  = 0.
\end{equation*}
As a result, the original equilibrium equation (\ref{eqn-pi_L-A_L-V=0}) still holds.

In the above setting, we estimate the difference $\pi - \pi_L$.
Let
\begin{equation*}
\|\pi - \pi_L\|_{\bv{g}} :=
\int_{x \in \bbR_+} \bv{g}(x)|\pi(dx) - \pi_L(dx)|,
\end{equation*}
where $\bv{g}: \bbX \to \bbR_+$ is an arbitrary Borel function belonging to both domains of $\scA$ and $\scA_L$. Let $g$ denote a function $\bbR_+ \to \bbR$ such that, for any $\bbA \in \calB(\bbR_+)$,
\begin{equation*}
g(\bbA) =
\left\{
\begin{array}{r@{~~~}l}
\bv{g}(\bbA),  & \pi(\bbA) - \pi_L(\bbA) \ge 0,
\\
-\bv{g}(\bbA), & \pi(\bbA) - \pi_L(\bbA) < 0,
\end{array}
\right.
\end{equation*}
which yields
\begin{equation}
\|\pi - \pi_L\|_{\bv{g}}
= \braket{ \pi - \pi_L, g}.
\label{eqn-distance-pi-pi_L}
\end{equation}

We now introduce the Poisson equation:
\[
-\scA_L h = g - \braket{\pi_L, g}.
\]
We then define $h_L^{(g)}:\bbR_+ \to \bbR$ as the standard solution of this Poisson equation; that is (see (\ref{Poisson-EQ}) and (\ref{defn-h^{(g)}})),
\begin{equation*}
h_L^{(g)}(x) 
= \EE_x \! \left[ \int_0^{\tau_{L,0}} g(X_L(t)) dt\right]
- \braket{\pi_L, g} \EE_x[\tau_{L,0}],\qquad x \in \bbR_+,
\end{equation*}
where $\tau_{L,0}:=\inf\{t > 0: X_L(t) = 0, X_L(t-) \neq 0\}$. By definition,
\begin{equation}
-\scA_L h_L^{(g)} = g - \braket{\pi_L, g}.
\label{Poisson-EQ-finite}
\end{equation}
Using (\ref{eqn-distance-pi-pi_L}) and (\ref{Poisson-EQ-finite}), we have
\begin{eqnarray}
\|\pi - \pi_L\|_{\bv{g}}
&=& \braket{\pi, g} - \braket{\pi_L, g}
\nonumber
\\
&=& \braket{\pi, g - \braket{\pi_L, g}\cdot 1_{\bbR_+}}
\nonumber
\\
&=& \braket{\pi, -\scA_L h_L^{(g)}}
\nonumber
\\
&=& \braket{\pi, (\scA -\scA_L) h_L^{(g)}},
\label{diff-|pi-pi_L|-00}
\end{eqnarray}
where the last equality holds because $\pi$ satisfies (\ref{eqn-pi*A=0}).

Let $|\scA - \scA_L|$ denote a generator such that
\begin{eqnarray}
|\scA - \scA_L| V(x)
&=& 
\left\{
\begin{array}{l@{~~~}l}
\lambda \dm\int_{L-x}^{\infty} H(dy)\{V(x+y)+V(x)\}, & 0\le x \le L,
\\
0, & x > L.
\end{array}
\right. \qquad
\label{eqn-(scA-scA_L)V_0}
\end{eqnarray}
It then follows from (\ref{defn-A-MG1}), (\ref{defn-A-MG1-finite-03}) and (\ref{eqn-(scA-scA_L)V_0}) that
\begin{equation*}
|(\scA - \scA_L) h_L^{(g)}| 
\le |\scA - \scA_L| h_L^{(g)}
\le |\scA - \scA_L|\, |h_L^{(g)}|.
\end{equation*}
Combining this and (\ref{diff-|pi-pi_L|-00}) results in
\begin{eqnarray}
\|\pi - \pi_L\|_{\bv{g}}
&\le& \braket{\pi,  |(\scA - \scA_L) h_L^{(g)}| }
\nonumber
\\
&\le& \braket{\pi,  |\scA - \scA_L|\, |h_L^{(g)}|}.
\label{diff-|pi-pi_L|}
\end{eqnarray}
Therefore, bounding $|h_L^{(g)}|$, we can obtain a bound for $\|\pi - \pi_L\|_{g}$. 

To achieve this, we assume that the $f$-modulated drift condition (Condition~\ref{cond-drift}) holds for $\bbC = \alpha = \{0\}$, where $V$ is increasing and differentiable (Indeed, we will later construct such $f$-modulated drift conditions in the present setting). It then follows from (\ref{defn-A-MG1}), (\ref{defn-A-MG1-finite-03}) and the increasingness of $V$ that 
\begin{equation}
\scA_L V \le \scA V \le  - f + b 1_{\{0\}}.
\label{ineqn-Qv-finite}
\end{equation}
Corollary~\ref{coro-explicit-bound}, together with (\ref{ineqn-Qv-finite}) and (\ref{eqn-pi(x)-MG1}), yields
\begin{equation}
|h_L^{(g)}| 
\le \left(1 + {b\pi(\{0\}) \over \inf_{y\in \bbX}f(y)} \right) V_0
= \left(1 + {b(1-\rho) \over \inf_{y\in \bbX}f(y)} \right) V_0
\quad \mbox{for all $|g| \le f$}.
\label{explicit-bounds-h_L-general}
\end{equation}

We now substitute (\ref{explicit-bounds-h_L-general}) into (\ref{diff-|pi-pi_L|}), which results in
\begin{eqnarray}
\|\pi - \pi_L\|_{\bv{g}}
&\le& \left(1 + {b(1-\rho) \over \inf_{y\in \bbR_+}f(y)} \right) 
\braket{\pi, (\scA - \scA_L) V_0},\quad 0 \le \bv{g} \le f.
\label{diff-|pi-pi_L|-02}
\end{eqnarray}
From (\ref{defn-V_0}) and (\ref{eqn-(scA-scA_L)V_0}), we also have
\begin{eqnarray*}
|\scA - \scA_L| V_0(x)
&=& 
\left\{
\begin{array}{l@{~~~}l}
\lambda \dm\int_{L-x}^{\infty} H(dy) \{ V_0(x+y) + V_0(x)\}, & 0\le x \le L,
\\
0, & x > L.
\end{array}
\right. \qquad
\end{eqnarray*}
Combining this and (\ref{diff-|pi-pi_L|-02}), and using (\ref{eqn-pi(x)-MG1}), we obtain the following bound: For all $0 \le \bv{g} \le f$,
\begin{eqnarray}
\|\pi - \pi_L\|_{\bv{g}}
&\le& 
\lambda 
\left(
1 + {b(1 - \rho) \over \inf_{y\in \bbR_+}f(y)} 
\right)
\sum_{m=0}^{\infty}(1 - \rho)\rho^m
\nonumber
\\
&& {}
\times  
\int_0^L H_{\re}^{*m}(dx)
\int_{L-x}^{\infty} H(dy) \{ V_0(x+y) + V_0(x) \}. \qquad
\label{bound-pi(g)-pi_L(g)}
\end{eqnarray}

In summary, we can obtain (\ref{explicit-bounds-h_L-general}) and thus (\ref{bound-pi(g)-pi_L(g)}), provided that the $f$-modulated drift condition (\ref{ineqn-Qv-finite}) holds for increasing and differentiable $V$. 
In the next subsection, we construct such drift conditions, and combining them with (\ref{explicit-bounds-h_L-general}), we derive some explicit bounds for $h_L^{(g)}$. Similarly, substituting the specified expressions of $V$ and $f$ into (\ref{bound-pi(g)-pi_L(g)}), we can obtain bounds for $\|\pi - \pi_L\|_{\bv{g}}$. However, those bounds would not be much simpler than the original bound (\ref{bound-pi(g)-pi_L(g)}). Thus, we omit the bounds to save space.

\subsection{Specific bounds for the standard solution of the Poisson equation}

We consider three cases: (i) the asymptotic tail decay of $H$ is light-tailed; (ii)  moderately exponential; and (iii) polynomial. For the three cases, we derive specific bounds (\ref{bound-M-GI-1-01}), (\ref{bound-moderate}) and (\ref{bound-Pareto}) in Sections~\ref{subsec-light-tailed}, \ref{subsec-moderate}, and \ref{subsec-Pareto}, respectively.

\subsubsection{Light-tailed case}\label{subsec-light-tailed}
\hfill

Suppose that $H$ is light-tailed, i.e., (\ref{cond-light-tailed}) holds.   Let $\sigma(\theta)$, $\theta < \ol{\theta}$, denote
\begin{equation}
\sigma(\theta) = -\lambda + \lambda \wh{H}(\theta).
\label{defn-sigma(theta)}
\end{equation}
Clearly, $\sigma(0) = 0$ and 
\[
\sigma'(0) = \lambda
\left. {d \over d\theta}\wh{H}(\theta) \right|_{\theta=0}
= \lambda /\mu = \rho < 1.
\]
Therefore, $\sigma(\theta) < \theta$ for some $\theta > 1$. 

We fix $\theta > 1$ such that $\sigma(\theta) < \theta$. We also fix
\begin{eqnarray}
V(x) &=& e^{\theta x},
\qquad x \in \bbR_+.
\label{eqn-v(x)-M-GI-1}
\end{eqnarray}
It then follows from (\ref{defn-A-MG1}), (\ref{defn-sigma(theta)}) and (\ref{eqn-v(x)-M-GI-1}) that
\begin{align*}
\scA V(0)
&=
-\lambda + \lambda \wh{H}(\theta) 
\nonumber
\\
&= \sigma(\theta)
\nonumber
\\
&= -(\theta - \sigma(\theta)) V(0) + \theta,
\end{align*}
and
\begin{align*}
\scA V(x)
&=
\left[
- \theta -\lambda + \lambda \wh{H}(\theta) 
\right] e^{\theta x}
\nonumber
\\
&= -(\theta - \sigma(\theta)) e^{\theta x}
\nonumber
\\
&= -(\theta - \sigma(\theta)) V(x) < 0, \qquad x > 0.
\end{align*}
These results lead to
\begin{equation*}
\scA V = -(\theta - \sigma(\theta)) V 
+ \theta 1_{\{0\}}.
\end{equation*}
Thus, Condition~\ref{cond-drift} holds with
\begin{eqnarray*}
b = \theta,
\quad
f = (\theta - \sigma(\theta)) V,
\quad
\bbC = \{0\},
\end{eqnarray*}
where $V$ is given in (\ref{eqn-v(x)-M-GI-1}). 
Therefore, (\ref{explicit-bounds-h_L-general}) yields, for $|g| \le (\theta - \sigma(\theta))V$,
\begin{eqnarray}
|h_L^{(g)}(x)|
&\le& 
\left(
1 + {\theta (1-\rho) \over \theta - \sigma(\theta)}
\right)
( e^{\theta x} - 1 ),\qquad x \in \bbR_+,
\label{bound-M-GI-1-01}
\end{eqnarray}
where $\sigma(\theta)$ is given in (\ref{defn-sigma(theta)}).

\subsubsection{Moderately exponential case}\label{subsec-moderate}
\hfill

We assume that, for some $\beta \in (0,1)$, $\gamma > 0$ and $C>0$,
\begin{eqnarray}
\ol{H}(x) 
&\le& C \exp\{-\gamma x^{\beta}\}\quad \mbox{for all $x \in \bbR_+$.}
\label{cond-ol{H}-moderate}
\end{eqnarray}
We then fix
\begin{equation}
V(x) = \exp\{\varepsilon (x+x_0)^{\beta} \},
\qquad x \in \bbR_+,
\label{defn-V-moderate}
\end{equation}
where $\varepsilon \in (0,\gamma)$ and 
\begin{equation}
x_0 \ge \left({1 - \beta \over \varep\beta}\right)^{1/\beta}.
\label{defn-x_0-moderate}
\end{equation}
The constraint (\ref{defn-x_0-moderate}) ensures that
\begin{equation}
V''(x) = \varep \beta (x+x_0)^{\beta-2} 
\left\{ \varep \beta (x+x_0)^{\beta} - (1 - \beta) \right\}V(x) \ge 0,
\qquad x \in\bbR_+.
\label{ineqn-V''>0}
\end{equation}

It follows from (\ref{cond-ol{H}-moderate}) and (\ref{defn-V-moderate}) that (\ref{cond-H}) and thus (\ref{defn-A-MG1-02}) hold (see Remark~\ref{rem-A-MG1}).
Substituting (\ref{defn-V-moderate}) into (\ref{defn-A-MG1-02}) with $x>0$, we have
\begin{eqnarray}
\scA V(x)
&=& - \beta \varep (x+x_0)^{\beta - 1}\exp\{\varepsilon (x+x_0)^{\beta} \}
\nonumber
\\
&&{}
+ \beta \varep\lambda \int_0^{\infty} \ol{H}(y) 
(x+x_0+y)^{\beta - 1}\exp\{\varepsilon (x+x_0+y)^{\beta} \}  dy,
~~ x > 0. \qquad~
\label{eqn-171124-01}
\end{eqnarray}
Note here that the following limit holds (which is proved in Appendix~\ref{proof-eqn-171124-02}): For $x \in \bbR_+$,
\begin{eqnarray}
&&
\lim_{x_0 \to \infty}
\lim_{\varepsilon\to0}
\int_0^{\infty} \ol{H}(y) 
{
(x+x_0+y)^{\beta - 1}\exp\{\varepsilon (x+x_0+y)^{\beta} \} 
\over
(x+x_0)^{\beta - 1}\exp\{\varepsilon (x+x_0)^{\beta} \} 
} dy
\nonumber
\\
&& \qquad 
= \int_0^{\infty} \ol{H}(y)  dy 
= \mu^{-1} < \lambda^{-1},
\label{eqn-171124-02}
\end{eqnarray}
where the last inequality is due to $\lambda / \mu = \rho < 1$.
Therefore, we can fix $x_0>0$, $\varep \in (0,\gamma)$ and $\wt{\rho} \in (\rho,1)$,
\begin{eqnarray}
&&
\lambda \int_0^{\infty} \ol{H}(y) 
(x+x_0+y)^{\beta - 1}\exp\{\varepsilon (x+x_0+y)^{\beta} \}  
 dy
\nonumber
\\
&& \qquad
\le 
\wt{\rho}(x+x_0)^{\beta - 1}\exp\{\varepsilon (x+x_0)^{\beta} \},
\qquad x \in \bbR_+.
\label{eqn-171124-03}
\end{eqnarray}
Applying (\ref{eqn-171124-03}) to the right hand side of (\ref{eqn-171124-01}), we obtain, for $x > 0$,
\begin{eqnarray}
\scA V(x)
&\le& - (1 - \wt{\rho}) 
\beta \varep (x+x_0)^{\beta - 1}\exp\{\varepsilon (x+x_0)^{\beta} \}
= - (1 - \wt{\rho})V'(x).
\label{ineqn-Av(x)-moderate}
\end{eqnarray}

We now fix $\bbC=\alpha = \{0\}$, and fix $b \ge 0$ such that
\begin{eqnarray}
b 
&\ge& (1 - \wt{\rho})V'(0) + \lambda \int_0^{\infty} \ol{H}(y)V'(y) dy. 
\label{ineqn-b-MG1}
\end{eqnarray}
It then follows from (\ref{defn-A-MG1-02}), (\ref{ineqn-Av(x)-moderate}) and (\ref{ineqn-b-MG1}) that
\begin{equation}
\scA V \le - (1 - \wt{\rho})V' + b 1_{\{0\}}.
\label{drift-cond-moderate}
\end{equation}
Thus, letting $f=(1 - \wt{\rho})V'$, and using the increasingness of $V'$ (due to (\ref{ineqn-V''>0})), we have
\[
\inf_{x \in \bbR_+}f(x)
= f(0) 
= (1 - \wt{\rho}) \beta \varep x_0^{\beta-1} \exp\{\varep x_0^{\beta}\}.
\] 
As a result, by (\ref{explicit-bounds-h_L-general}), we obtain, for $|g| \le (1-\rho)V'$,
\begin{eqnarray}
|h_L^{(g)}(x)| 
&\le& \left(
1 + {b(1 - \rho) 
\over (1 - \wt{\rho}) \beta \varep x_0^{\beta-1} \exp\{\varep x_0^{\beta}\}
} 
\right) 
\nonumber
\\
&& {} \times
\left[\exp \{\varep(x+x_0)^{\beta}\} - \exp\{\varep x_0^{\beta}\} \right],
\qquad x \in \bbR_+, 
\label{bound-moderate}
\end{eqnarray}
where $x_0 > 0$, $\varep \in (0,\gamma)$ and $\wt{\rho} \in (\rho,1)$ are constants satisfying (\ref{defn-x_0-moderate}) and (\ref{eqn-171124-03}).

\subsubsection{Polynomial case}\label{subsec-Pareto}

We assume that, for some $\kappa > 1$ and $C>0$,
\begin{equation}
\ol{H}(x) \le C (x+1)^{-\kappa}\quad \mbox{for all $x \in \bbR_+$.}
\label{cond-ol{H}-Pareto}
\end{equation}
We then fix
\begin{equation}
V(x) = (x+x_0)^{\wt{\kappa}},\qquad x \in \bbR_+,
\label{defn-V-Pareto}
\end{equation}
where $\wt{\kappa} \in (1,\kappa)$ and $x_0 \ge 1$. As in Section~\ref{subsec-moderate}, we can use (\ref{defn-A-MG1-02}). Thus, substituting (\ref{defn-V-Pareto}) into (\ref{defn-A-MG1-02}) with $x>0$, we have
\begin{equation}
\scA V(x)
= - \wt{\kappa} (x+x_0)^{\wt{\kappa}-1}
+ \wt{\kappa} \lambda \int_0^{\infty}\ol{H}(y)(x+x_0+y)^{\wt{\kappa}-1},
\qquad x > 0.
\label{eqn-scrA*V-Pareto}
\end{equation}
We also obtain the following limit (which is proved in Appendix~\ref{proof-eqn-171124-05}):
\begin{equation}
\lim_{x_0 \to \infty}
\int_0^{\infty}\ol{H}(y)
{ (x+x_0+y)^{\wt{\kappa}-1} \over (x+x_0)^{\wt{\kappa}-1} }
= \mu^{-1}, \qquad x \in \bbR_+,
\label{eqn-171124-05}
\end{equation}
Therefore, we can fix $x_0 \ge 1$ and $\wt{\rho} \in (\rho,1)$ such that
\begin{equation}
\lambda \int_0^{\infty}\ol{H}(y)
(x+x_0+y)^{\wt{\kappa}-1} \le \wt{\rho}(x+x_0)^{\wt{\kappa}-1}, \qquad x \in \bbR_+.
\label{eqn-171124-06}
\end{equation}
Substituting (\ref{eqn-171124-06}) into (\ref{eqn-scrA*V-Pareto}) yields
\begin{equation*}
\scA V(x)
\le - (1 - \wt{\rho})\wt{\kappa} (x+x_0)^{\wt{\kappa}-1}
=  -(1 - \wt{\rho})V'(x),
\qquad x > 0,
\end{equation*}
which is an inequality of the same type as (\ref{ineqn-Av(x)-moderate}) in Section~\ref{subsec-moderate}. Thus, (\ref{drift-cond-moderate}) holds for $\bbC=\alpha = \{0\}$ and $b \ge 0$ satisfying (\ref{ineqn-b-MG1}). 
Consequently, following the derivation of the bound (\ref{bound-moderate}), we obtain, for $|g| \le (1-\rho)V'$,
\begin{eqnarray}
|h_L^{(g)}(x)| 
&\le& \left(
1 + {b(1 - \rho)
\over (1 - \wt{\rho}) \wt{\kappa} x_0^{\wt{\kappa}-1}
} 
\right) 
\left[(x+x_0)^{\kappa} - x_0^{\kappa} \right],
\qquad x \in \bbR_+,
\label{bound-Pareto}
\end{eqnarray}
where $x_0 \ge 1$ and $\wt{\rho} \in (\rho,1)$ are constants satisfying (\ref{eqn-171124-06}).

\appendix

\section {Proofs}

\subsection{Proof of (\ref{eqn-pi*A=0})}\label{appen-proof-eqn-pi*A=0}

Let $R$ denote a resolvent kernel such that
\[
R(x,\bbA) = \int_0^{\infty} e^{-t} P^t(x,\bbA) dt,
\qquad x \in \bbX,~\bbA\in\calB(\bbX).
\]
It follows from \cite[Eq.~(13) and Lemma~3.1]{Glyn96} that $\scA$ and $R$ are commute, i.e.,
\begin{equation}
\scA R = R \scA.
\label{eqn-AR=RA}
\end{equation}
Let $V = (I - R)^{-1}F$, where $F$ is an arbitrary function in the domain of $\scA$. Clearly, $F = (I-R)V$. Using this together with (\ref{eqn-AR=RA}) and $\pi R = \pi$, we have
\begin{eqnarray*}
\pi \scA F 
&=& \pi \scA (I-R)V
\nonumber
\\
&=& \pi \scA V - \pi \scA R  V
\nonumber
\\
&=& \pi \scA V - \pi R \scA  V
\nonumber
\\
&=& \pi \scA V - \pi \scA  V = 0,
\end{eqnarray*}
which implies that (\ref{eqn-pi*A=0}) holds.

\subsection{Sufficient condition for Condition~\ref{cond-petite}}\label{proof-lem-petite}

The following lemma provides a sufficient condition for Condition~\ref{cond-petite}.
\begin{lem}\label{lem-petite}
Suppose that Conditions~\ref{cond-drift} and \ref{cond-atom} are satisfied. If
the petite set $\bbC$ (which appears in Condition~\ref{cond-drift}) is finite, then Condition~\ref{cond-petite} holds.
\end{lem}

\begin{proof}
Condition~\ref{cond-atom} shows that the set $\alpha$ is an accessible atom. Therefore, for each $x \in \bbC$, there exist positive numbers $t_x,\varepsilon_x > 0$ such that
\begin{eqnarray*}
\PP(X(t_x) \in \alpha \mid X(0) = x) \ge \varepsilon_x,
\end{eqnarray*}
and thus
\begin{equation}
\PP(X(t_x) \in \alpha \mid X(0) = x) 
\ge \min_{y \in \bbC}\varepsilon_y =: \varepsilon_{\ast} 
\quad \mbox{for all $x \in \bbC$.}
\label{eqn-171101-02}
\end{equation}
Furthermore, by the Markov property, there exists some $c_{\ast} > 0$ such that, for all $t \in \bbR_+$ and $x \in \bbC$,
\begin{eqnarray}
\qquad
\PP(X(u+t_x) \in \alpha, \forall u \in [0,t] \mid X(t_x) \in \alpha) 
= e^{-c_{\ast} t}.
\label{eqn-171101-03}
\end{eqnarray}
Combining (\ref{eqn-171101-02}) and (\ref{eqn-171101-03}) yields
\begin{eqnarray}
\PP(X(t+t_x) \in \alpha \mid X(0) = x)
&\ge& \varepsilon_{\ast} e^{-c_{\ast} t}
\quad \mbox{for all $t \in \bbR_+$ and $x \in \bbC$.}
\label{eqn-171101-04}
\end{eqnarray}

We now fix
\begin{eqnarray*}
T 
&=& \max_{x\in\bbC}t_x,
\qquad
\xi_T
=\varepsilon_{\ast} e^{-c_{\ast}(T-t_x)} > 0.
\end{eqnarray*}
It then follows from (\ref{eqn-171101-04}) that, for all $x \in \bbC$,
\begin{eqnarray*}
\PP(X(T) \in \alpha \mid X(0) = x)
&=& \PP(X(T-t_x + t_x) \in \alpha \mid X(0) = x)
\nonumber
\\
&\ge& \varepsilon_{\ast} e^{-c_{\ast}(T-t_x)}
= \xi_T,
\end{eqnarray*}
which implies that (\ref{cond-petite-01}) holds. \qed
\end{proof}

\subsection{Proof of Theorem~\ref{thm-bound-h}}\label{proof-main-thm-general}

Premultiplying by $\pi$ both sides of (\ref{ineqn-Qv-general}) yields
\[
\braket{\pi,  f} \le b \pi(\bbC),\qquad |g| \le f.
\]
Thus,
\[
|\braket{\pi,g}|   \le 
\braket{\pi,  |g|} \le  
\braket{\pi,  f} \le b \pi(\bbC),\qquad |g| \le f.
\]
Applying this inequality to the first bound (\ref{bound-h-general-stronger}) results in the second one (\ref{bound-h-general-weaker}). Therefore, we prove (\ref{bound-h-general-stronger}).

We first note that
\begin{equation}
\EE_x[\tau_{\alpha}]
\le {1 \over \inf_{y \in \bbX} f(y) }
\EE_x \! \left[
\int_0^{\tau_{\alpha}} f(X(t)) dt
\right].
\label{eqn-180620-01}
\end{equation}
From (\ref{defn-h^{(g)}}) and (\ref{eqn-180620-01}), we have
\begin{eqnarray}
|h^{(g)}(x)| 
&\le& \EE_x \! \left[ \int_0^{\tau_{\alpha}} |g(X(t))| dt\right]
+ |\braket{\pi,  g}| \,\EE_x[\tau_{\alpha}]
\nonumber
\\
&\le& \EE_x \! \left[ \int_0^{\tau_{\alpha}} f(X(t)) dt\right]
+ |\braket{\pi,  g}|\, \EE_x[\tau_{\alpha}]
\nonumber
\\
&\le& \left(1 + { |\braket{\pi,  g}| \over \inf_{y \in \bbX} f(y)}  \right)
\EE_x \! \left[ \int_0^{\tau_{\alpha}} f(X(t)) dt\right],
\qquad x \in \bbX.
\label{ineqn-h^{(g)}}
\end{eqnarray}
From Dynkin's formula (see, e.g., \cite{Davi93}) and (\ref{ineqn-Qv-general}), we also have
\begin{eqnarray*}
\EE_x[V(X(\tau_{\alpha}))]
&=& V(x) + \EE_x \! \left[ \int_0^{\tau_{\alpha}} \scA V(X(t))dt  \right]
\nonumber
\\
&\le& V(x) - \EE_x \! \left[ \int_0^{\tau_{\alpha}} f(X(t))dt  \right]
+ b \EE_x \! \left[ \int_0^{\tau_{\alpha}} 1_{\bbC}(X(t))dt  \right].
\end{eqnarray*}
Using this result together with $\EE_x[V(X(\tau_{\alpha}))] \ge \inf_{y\in\alpha}V(y)$, we obtain
\begin{eqnarray}
\EE_x \! \left[ \int_0^{\tau_{\alpha}} f(X(t))dt  \right]
&\le& V(x) - \inf_{y\in\alpha}V(y) 
 + b \EE_x \! \left[ \int_0^{\tau_{\alpha}} 1_{\bbC}(X(t))dt  \right]
\nonumber
\\
&=& V_0(x) + b \EE_x \! \left[ \int_0^{\tau_{\alpha}} 1_{\bbC}(X(t))dt  \right],\qquad x \in \bbX,
\label{eqn-171026-01}
\end{eqnarray}
where the last equality holds due to $V_0 = V-\inf_{y\in\alpha}V(y)$. Furthermore, substituting (\ref{eqn-171026-01}) into (\ref{ineqn-h^{(g)}}) yields
\begin{eqnarray}
|h^{(g)}(x)| 
&\le&  \left(1 + { |\braket{\pi,  g}| \over \inf_{y \in \bbX} f(y)}  \right)
\nonumber
\\
&& {} \times 
\left(
V_0(x) + b \EE_x \! \left[ \int_0^{\tau_{\alpha}} 1_{\bbC}(X(t))dt  \right]
\right),\qquad x \in \bbX.
\label{ineqn-h^{(g)}-02}
\end{eqnarray}

To complete the proof, we estimate $\EE_x[ \int_0^{\tau_{\alpha}} 1_{\bbC}(X(t))dt]$. Lemma~\ref{lem-xi_T} yields
\begin{equation*}
1_{\bbC}(x)
\le {P^T(x,\alpha) \over \xi_T},
\qquad x \in \bbX,~T \ge T_*.
\end{equation*}
Using this inequality, we obtain
\begin{eqnarray}
\EE_x\! \left[ \int_0^{\tau_{\alpha}} 1_{\bbC}(X(t))dt  \right]
&\le& 
{ 1 \over \xi_T }
\EE_x \!\left[ 
\int_0^{\tau_{\alpha}} P^T(X(t),\alpha) dt
\right]
\nonumber
\\
&=& { 1 \over \xi_T }
\EE_x\! \left[
\int_0^{\tau_{\alpha}} \EE\!\left[ 1_{\alpha}(X(t+T)) \mid X(t)\right]dt
\right]
\nonumber
\\
&=& { 1 \over \xi_T }
\EE_x\! \left[
\int_0^{\tau_{\alpha}} 1_{\alpha}(X(t+T)) dt
\right],
\qquad x \in \bbX. 
\label{ineqn-02}
\end{eqnarray}
Note here that
\begin{align*}
&&&&
\int_0^{\tau_{\alpha}} 1_{\alpha}(X(t+T)) dt &\le 
\int_0^T 1_{\alpha}(X(t+T)) dt \le T, & 
&\mbox{if $T \ge \tau_{\alpha}$},
&&&&
\\
&&&&
\int_0^{\tau_{\alpha}} 1_{\alpha}(X(t+T)) dt &
= \int_{\tau_{\alpha}}^{\tau_{\alpha} + T} 1_{\alpha}(X(t)) dt \le T, & 
&\mbox{if $T < \tau_{\alpha}$},
&&&&
\end{align*}
which lead to
\[
\EE_x\! \left[
\int_0^{\tau_{\alpha}} 1_{\alpha}(X(t+T)) dt
\right] \le T.
\]
Thus, from (\ref{ineqn-02}), we have
\begin{eqnarray}
\EE_x\! \left[ \int_0^{\tau_{\alpha}} 1_{\bbC}(X(t))dt  \right]
\le { T \over \xi_T },
\qquad x \in \bbX. 
\label{ineqn-03}
\end{eqnarray}
Substituting (\ref{ineqn-03}) into (\ref{ineqn-h^{(g)}-02}) results in (\ref{bound-h-general-stronger}). 

\subsection{Proof of (\ref{eqn-171124-02})}\label{proof-eqn-171124-02}

Since $0 < \beta < 1$, the following holds for $x \in \bbR_+$, $x_0 > 0$ and $\varep \in (0,\gamma)$:
\begin{eqnarray}
&&
\int_0^{\infty}\ol{H}(y)
{
(x+x_0+y)^{\beta - 1}\exp\{\varepsilon (x+x_0+y)^{\beta} \} 
\over
(x+x_0)^{\beta - 1}\exp\{\varepsilon (x+x_0)^{\beta} \} 
} dy
\nonumber
\\
&& \qquad 
\le 
\int_0^{\infty}\ol{H}(y)
{
\exp\{\varepsilon (x+x_0+y)^{\beta} \} 
\over
\exp\{\varepsilon (x+x_0)^{\beta} \} 
} dy
\nonumber
\\
&& \qquad
= 
\int_0^{\infty}\ol{H}(y)
\exp\left\{
\varepsilon (x+x_0)^{\beta}
\left[
\left( 1 + {y \over x+x_0} \right)^{\beta} - 1
\right]
\right\} dy.
\label{eqn-171124-04}
\end{eqnarray}
Furthermore, for any $0 < \beta < 1$,
\[
\left( 1 + {y \over x+x_0} \right)^{\beta} 
\le 1 + \left(  {y \over x+x_0} \right)^{\beta},
\qquad x,y \in\bbR_+,~x_0 > 0.
\]
Applying this to the right hand side of (\ref{eqn-171124-04}) yields
\begin{eqnarray*}
&&
\int_0^{\infty}\ol{H}(y)
{
(x+x_0+y)^{\beta - 1}\exp\{\varepsilon (x+x_0+y)^{\beta} \} 
\over
(x+x_0)^{\beta - 1}\exp\{\varepsilon (x+x_0)^{\beta} \} 
} dy
\nonumber
\\
&& \qquad 
\le
\int_0^{\infty}\ol{H}(y)\exp\{ \varepsilon y^{\beta} \} dy
\nonumber
\\
&& \qquad 
\le
C\int_0^{\infty}\exp\{ -(\gamma - \varep) y^{\beta} \} dy < \infty
\quad \mbox{for all $\varep \in (0,\gamma)$ and $x_0 > 0$},
\end{eqnarray*}
where the second inequality holds due to (\ref{cond-ol{H}-moderate}). Therefore, by the dominated convergence theorem, we obtain 
\begin{eqnarray*}
&&
\lim_{x_0 \to \infty}
\lim_{\varep \to 0}
\int_0^{\infty}\ol{H}(y)
{
(x+x_0+y)^{\beta - 1}\exp\{\varepsilon (x+x_0+y)^{\beta} \} 
\over
(x+x_0)^{\beta - 1}\exp\{\varepsilon (x+x_0)^{\beta} \} 
} dy
\nonumber
\\
&& \quad
= 
\int_0^{\infty}\ol{H}(y)
\lim_{x_0 \to \infty}
\lim_{\varep \to 0}
\left[
{
(x+x_0+y)^{\beta - 1}
\over
(x+x_0)^{\beta - 1}
} 
{
\exp\{\varepsilon (x+x_0+y)^{\beta} \} 
\over
\exp\{\varepsilon (x+x_0)^{\beta} \} 
} 
\right]dy
\nonumber
\\
&& \quad
= 
\int_0^{\infty}\ol{H}(y)dy
=\mu^{-1},
\end{eqnarray*}
which shows that (\ref{eqn-171124-02}) holds.

\subsection{Proof of (\ref{eqn-171124-05})}\label{proof-eqn-171124-05}

It follows from (\ref{cond-ol{H}-Pareto}), $1 < \wt{\kappa} < \kappa$ and $x_0 > 1$ that, for $x \in \bbR_+$, 
\begin{eqnarray*}
&&
\int_0^{\infty} \ol{H}(y)
{ (x + x_0 + y)^{\wt{\kappa}-1} \over (x + x_0)^{\wt{\kappa}-1} } dy
\nonumber
\\
&& \qquad
\le C\int_0^{x_0} (y+1)^{-\kappa}
\left( 1 + { y \over x + x_0} \right)^{\wt{\kappa}-1} dy
\nonumber
\\
&& \qquad~~~~ {}
+ C\int_{x_0}^{\infty} (y+1)^{-\kappa}
\left( 1 + { y \over x + x_0} \right)^{\wt{\kappa}-1} dy
\nonumber
\\
&& \qquad
\le C\int_0^{x_0} (y+1)^{-\kappa} \cdot 2^{\wt{\kappa} - 1} dy
+ C\int_{x_0}^{\infty} (y+1)^{-\kappa}
\left( 1 + y \right)^{\wt{\kappa}-1} dy
\nonumber
\\
&& \qquad
\le 2^{\wt{\kappa} - 1}C \int_0^{x_0} (y+1)^{-\kappa} dy
+   2^{\wt{\kappa} - 1}C \int_{x_0}^{\infty} (y+1)^{-\kappa+\wt{\kappa}-1} dy
\nonumber
\\
&& \qquad
\le  
2^{\wt{\kappa} - 1}C
\int_0^{\infty} (y+1)^{-\kappa+\wt{\kappa}-1}  dy
= {2^{\wt{\kappa} - 1}C \over \kappa - \wt{\kappa}} < \infty.
\end{eqnarray*}
Therefore, by the dominated convergence theorem, we have
\begin{eqnarray*}
\lim_{x_0\to\infty}
\int_0^{\infty} \ol{H}(y)
{ (x + x_0 + y)^{\wt{\kappa}-1} \over (x + x_0)^{\wt{\kappa}-1} } dy
&=& \int_0^{\infty} \ol{H}(y)\lim_{x_0\to\infty}
\left( { x + x_0 + y \over x + x_0} \right)^{\wt{\kappa}-1} dy
\nonumber
\\
&=& \int_0^{\infty} \ol{H}(y) = \mu^{-1},
\end{eqnarray*}
which shows that (\ref{eqn-171124-05}) holds.


\section*{Acknowledgments}
This research was supported in part by JSPS KAKENHI Grant Number JP18K11181.

%
%
%
\bibliographystyle{plain} 

\begin{thebibliography}{10}

\bibitem{Asmu03}
S.~Asmussen.
\newblock {\em Applied Probability and Queues}.
\newblock Springer, New York, {S}econd edition, 2003.

\bibitem{Asmu94-QUESTA}
S.~Asmussen and M.~Bladt.
\newblock Poisson's equation for queues driven by a {Markovian} marked point
  process.
\newblock {\em Queueing Systems}, 17(1--2):235--274, 1994.

\bibitem{Bini16}
D.~Bini, S.~Dendievel, G.~Latouche, and B.~Meini.
\newblock General solution of the {Poisson} equation for quasi-birth-and-death
  processes.
\newblock {\em SIAM Journal on Applied Mathematics}, 76(6):2397--2417, 2016.

\bibitem{Cao98}
Xi-Ren Cao.
\newblock The relations among potentials, perturbation analysis, and {Markov}
  decision processes.
\newblock {\em Discrete Event Dynamic Systems}, 8(1):71--87, 1998.

\bibitem{Cao97}
Xi-Ren Cao and Han-Fu Chen.
\newblock Perturbation realization, potentials, and sensitivity analysis of
  {Markov} processes.
\newblock {\em IEEE Transactions on Automatic Control}, 42(10):1382--1393,
  1997.

\bibitem{Davi93}
M.~H.~A. Davis.
\newblock {\em Markov Models and Optimization}.
\newblock Chapman \& Hall/CRC, London, 1993.

\bibitem{Dend13}
S.~Dendievel, G.~Latouche, and Y.~Liu.
\newblock {Poisson's} equation for discrete-time quasi-birth-and-death
  processes.
\newblock {\em Performance Evaluation}, 70(9):564--577, 2013.

\bibitem{Dynk65}
E.~B. Dynkin.
\newblock {\em Markov Processes}, volume~1.
\newblock Springer, Berlin, 1965.

\bibitem{Ethi05}
S.~N. Ethier and T.~G. Kurtz.
\newblock {\em Markov Processes: Characterization and Convergence}.
\newblock John Wiley \& Sons, Hoboken, NJ, 2005.

\bibitem{Glyn94}
P.~W. Glynn.
\newblock Poisson's equation for the recurrent {M/G/1} queue.
\newblock {\em Advances in Applied Probability}, 26(4):1044--1062, 1994.

\bibitem{Glyn96}
P.~W. Glynn and S.~P. Meyn.
\newblock A {Liapounov} bound for solutions of the {Poisson} equation.
\newblock {\em The Annals of Probability}, 24(2):916--931, 1996.

\bibitem{Gros08}
D.~Gross, J.~F. Shortle, J.~M. Thompson, and C.~M. Harris.
\newblock {\em Fundamentals of Queueing Theory}.
\newblock John Wiley \& Sons, Hoboken, NJ, 4th edition, 2008.

\bibitem{Horn13}
R.~A. Horn and C.~R. Johnson.
\newblock {\em Matrix Analysis}.
\newblock Cambridge University Press, Cambridge, 2nd edition, 2013.

\bibitem{Kont16}
I.~Kontoyiannis and S.~P. Meyn.
\newblock On the $f$-norm ergodicity of {Markov} processes in continuous time.
\newblock {\em Electronic Communications in Probability}, 21:paper no. 77,
  1--10, 2016.

\bibitem{Lawl06}
Gregory~F. Lawler.
\newblock {\em Introduction to Stochastic Processes}.
\newblock Chapman \& Hall/CRC, Boca Raton, FL, 2nd edition, 2006.

\bibitem{Lax07}
P.~D. Lax.
\newblock {\em Linear Algebra and Its Applications}.
\newblock John Wiley \& Sons, Hoboken, NJ, 2nd edition, 2007.

\bibitem{LiuYuan15}
Y.~Liu.
\newblock Perturbation analysis for continuous-time {Markov} chains.
\newblock {\em Science China Mathematics}, 58(12):2633--2642, 2015.

\bibitem{LiuYuan18-AAP}
Y.~Liu and W.~Li.
\newblock Error bounds for augmented truncation approximations of {Markov}
  chains via the perturbation method.
\newblock {\em Advances in Applied Probability}, 50(2):645--669, 2018.

\bibitem{LLM2018}
Y.~Liu, W.~Li, and Hiroyuki Masuyama.
\newblock Error bounds for augmented truncation approximations of
  continuous-time {Markov} chains.
\newblock {\em Operations Research Letters}, 46(4):409--413, 2018.

\bibitem{LiuYuan14}
Y.~Liu, P.~Wang, and Y.~Xie.
\newblock Deviation matrix and asymptotic variance for {GI/M/1}-type {Markov}
  chains.
\newblock {\em Frontiers of Mathematics in China}, 9(4):863--880, 2014.

\bibitem{Loyn62}
R.~M. Loynes.
\newblock The stability of a queue with non-independent inter-arrival and
  service times.
\newblock {\em Mathematical Proceedings of the Cambridge Philosophical
  Society}, 58(3):497--520, 1962.

\bibitem{Luca90}
D.~M. Lucantoni, K.~S. Meier-Hellstern, and M.~F. Neuts.
\newblock A single-server queue with server vacations and a class of
  non-renewal arrival processes.
\newblock {\em Advances in Applied Probability}, 22(3):676--705, 1990.

\bibitem{Mako92}
A.~M. Makowski and A.~Shwartz.
\newblock Stochastic approximations and adaptive control of a discrete-time
  single-server network with random routing.
\newblock {\em SIAM Journal on Control and Optimization}, 30(6):1476--1506,
  1992.

\bibitem{Mako02}
A.~M. Makowski and A.~Shwartz.
\newblock The {Poisson} equation for countable {Markov} chains: Probabilistic
  methods and interpretations.
\newblock In E.~A. Feinberg and A.~Shwartz, editors, {\em Handbook of Markov
  Decision Processes. International Series in Operations Research \& Management
  Science, vol. 40}, pages 269--303. Springer, Boston, MA, 2002.

\bibitem{Marc06}
M.~B. Marcus and J.~Rosen.
\newblock {\em Markov Processes, Gaussian Processes, and Local Times}.
\newblock Cambridge University Press, Cambridge, 2006.

\bibitem{Masu17-JORSJ}
H.~Masuyama.
\newblock Error bounds for last-column-block-augmented truncations of
  block-structured {Markov} chains.
\newblock {\em Journal of the Operations Research Society of Japan},
  60(3):271--320, 2017.

\bibitem{Meti84}
M.~Metivier and P.~Priouret.
\newblock Applications of a {Kushner} and {Clark} lemma to general classes of
  stochastic algorithms.
\newblock {\em IEEE Transactions on Information Theory}, 30(2):140--151, 1984.

\bibitem{Meyn93-II}
S.~P. Meyn and R.~L. Tweedie.
\newblock Stability of {Markov} processes {II}: {Continuous}-time processes and
  sampled chains.
\newblock {\em Advances in Applied Probability}, 25(3):487--517, 1993.

\bibitem{Meyn93-III}
S.~P. Meyn and R.~L. Tweedie.
\newblock Stability of {Markovian} processes {III}: {Foster-Lyapunov} criteria
  for continuous-time processes.
\newblock {\em Advances in Applied Probability}, 25(3):518--548, 1993.

\bibitem{Meyn93-proc}
S.~P. Meyn and R.~L. Tweedie.
\newblock A survey of {Foster-Lyapunov} techniques for general state space
  {Markov} processes.
\newblock In {\em Proceedings of the Workshop on Stochastic Stability and
  Stochastic Stabilization}, Metz, France, June 1993.

\bibitem{Meyn09}
S.~P. Meyn and R.~L. Tweedie.
\newblock {\em Markov Chains and Stochastic Stability}.
\newblock Cambridge University Press, Cambridge, 2nd edition, 2009.

\bibitem{Perr01}
D.~Perry, W.~Stadje, and S.~Zacks.
\newblock The {M/G/1} queue with finite workload capacity.
\newblock {\em Queueing Systems}, 39(1):7--22, 2001.

\bibitem{Tijm03}
H.~C. Tijms.
\newblock {\em A First Course in Stochastic Models}.
\newblock John Wiley \& Sons, Chichester, UK, 2003.

\end{thebibliography}


\end{document}